\documentclass[letterpaper, 10 pt, conference]{ieeeconf}  

\IEEEoverridecommandlockouts

\usepackage{amsmath,amssymb,amsfonts}
 \usepackage{algorithm}
\usepackage{algorithmicx}
\usepackage{algpseudocode}
\usepackage{graphicx}
\usepackage{subcaption}
\usepackage{textcomp}
\usepackage{dsfont}
\let\labelindent\relax
\usepackage{enumitem}
\usepackage{cite}
\usepackage{color}

\newtheorem{theorem}{\bf Theorem}

\newtheorem{remark}{\bf Remark}

\newtheorem{lemma}{\bf Lemma}

\def\QED{~\rule[-1pt]{5pt}{5pt}\par\medskip}

\title{\LARGE \bf
Covert Vehicle Misguidance and Its Detection: \\ A Hypothesis Testing Game over Continuous-Time Dynamics
}

\author{Takashi Tanaka$^{1}$ \and Kenji Sawada$^{2}$ \and Yohei Watanabe$^{3}$ \and Mitsugu Iwamoto$^{3}$
\thanks{$^{1}$ School of Aeronautics and Astronautics, School of Electrical and Computer Engineering, Purdue University
(e-mail: \texttt{tanaka16@purdue.edu}).}%
\thanks{$^{2}$The Info-Powered Energy System Research Center, The University of Electro-Communications
        (e-mail: {\tt\small knj.sawada@uec.ac.jp}).}%
\thanks{$^{3}$Department of Informatics, The University of Electro-Communications
       (e-mail:  {\tt\small \{watanabe,mitsugu\}@uec.ac.jp}).}%
}

\begin{document}

\maketitle
\thispagestyle{empty}
\pagestyle{empty}

\begin{abstract}

We formulate a stochastic zero-sum game over continuous-time dynamics to analyze the competition between the attacker, who tries to covertly misguide the vehicle to an unsafe region, versus the detector, who tries to detect the attack signal based on the observed trajectory of the vehicle.
Based on Girsanov's theorem and the generalized Neyman-Pearson lemma, we show that a constant bias injection attack as the attacker's strategy and a likelihood ratio test as the detector's strategy constitute the unique saddle point of the game.
We also derive the first-order and the second-order exponents of the type II error as a function of the data length.
\end{abstract}

\section{Introduction}
\label{sec:introduction}
False data injection (FDI) attacks are widely recognized as major threats to control systems. In \cite{bhatti2017hostile}, the authors performed a field experiment to misguide a 65-meter yacht to its unintended destination via GPS spoofing and demonstrated the vulnerability of modern maritime vessels to deceptive sensor data injection. 
In their experiment, the authors showed that a GPS deception attack, if carefully designed, can be disguised as the effects of natural disturbances such as slowly changing ocean currents and winds, and is difficult to detect unless the controller has an alternative source of reliable sensor data (e.g., radar and visual bearing).

In many circumstances, system faults (including malicious attacks) must be detected and isolated by continuous monitoring of the sensor readouts. 
In the vessel misguidance example  \cite{bhatti2017hostile}, the spoofed GPS signal may be distinguished from the natural background noise by an appropriate statistical test. 
However, knowing that the system is continuously monitored, a rational attacker will conduct a covert attack, maximizing the attack's impact while avoiding detection.
Hence, a zero-sum game arises between the attacker and the detector, where the detector's purpose is to design the ``most effective" statistical test for attack detection, whereas the attacker tries to inject the ``most stealthy" attack signal.

Similar games between the attacker and the detector have been studied by many authors in the systems and control community.
For example, the works \cite{bai2017data, guo2018worst,shang2021worst} adopted the hypothesis testing theory to characterize covert FDI attacks against control systems. 
Invoking Stein's lemma, \cite{bai2017data} introduced the notion of $\epsilon$-stealthiness as measured by relative entropy. 
The worst-case degradation of linear control systems attainable by $\epsilon$-stealthy attacks was studied in \cite{guo2018worst,shang2021worst}.
Sequential and composite hypothesis testing frameworks have also been proposed (e.g., \cite{zhang2017statistical,salimi2019sequential} to name a few) for anomaly detection.

Despite recent progress, existing applications of hypothesis testing frameworks to control systems are limited to discrete-time settings.
The goal of this paper is to broaden the scope of the literature by formulating the aforementioned zero-sum game in continuous time based on the generalized Neyman-Pearson theory \cite{cvitanic2001generalized}.
We make the following methodological contributions:
\begin{enumerate}[leftmargin=*]
\item We propose a novel zero-sum game formulation to model the competition between the attacker and the detector over continuous-time dynamics.
Instead of taking relative entropy as a stealthiness measure for granted (the operational meaning of relative entropy in continuous-time hypothesis testing scenarios is not well-established in the literature), we use more fundamental quantities, such as probabilities of type I and type II errors and the probability of successful attacks, to formulate the game.
While the game considered in this paper is simple, the results and methodologies we present are canonical and allow for various generalizations in future studies.
\item We show that a constant bias injection attack as the attacker's strategy and a likelihood ratio test as the detector's strategy constitute the unique saddle point of the game.
The proof is based on Girsanov's theorem and the generalized Neyman-Pearson lemma.
\item We analyze the exponent of the type II error as a function of the horizon length of the game and show that the first-order asymptote coincides with the relative entropy.
This result is reminiscent of classical Stein's lemma. We also quantify the second-order asymptote, providing a tighter estimate of the error probability in the finite horizon length regime.
\end{enumerate}

\subsubsection*{Notation}
The normal distribution with mean $m$ and covariance $\sigma^2$ is denoted by $\mathcal{N}(m,\sigma^2)$, and the cumulative distribution function of the standard normal distribution is denoted by
$\Phi(x)=\frac{1}{\sqrt{2\pi}} \int_{-\infty}^x \exp(-\frac{t^2}{2})dt$.
We write $[ \; \cdot \;]^+:=\max\{0,\cdot\}$.
$\mathds{1}_{\{\cdot\}}$ represents the indicator function.
$C[0,T]$ is the space of continuous functions $x: [0,T]\rightarrow \mathbb{R}$.
The Radon-Nikodym derivative of a probability measure $\mu$ with respect to a probability measure 
$\nu$ is denoted by
$\frac{d\mu}{d\nu}$.

\section{Problem formulation}
\label{sec:problem}
Inspired by the vessel misguidance \cite{bhatti2017hostile}, we formulate a stochastic zero-sum game modeling the competition between an attacker, who tries to covertly misguide the vehicle to an unsafe region, versus a detector, who tries to detect the existence of the attack based on the observed trajectory of the vehicle.
For simplicity, we model the trajectory of the vehicle (deviation from the nominal trajectory) as a continuous-time, scalar-valued Ito process $x_t$  over the time interval $0 \leq t \leq T$ defined by the following stochastic differential equation:
\begin{equation}
\label{eq:sde}
dx(t) = \theta(t)dt + dw(t), x(0)=0.
\end{equation}
Here, $w(t)$ is the standard Brownian motion in the underlying probability space $(\Omega, \mathcal{F}, \mu)$. For each $0 \leq t \leq T$, we denote by $\mathcal{F}(t)\subseteq \mathcal{F}$ the filtration of the process $w(t)$. 
We call the drift term $\theta: [0,T]\rightarrow \mathbb{R}$ the attack signal, which is chosen by the attacker. The attack signal $\theta$ is assumed to be a Borel measurable function such that $\int_0^T |\theta(t)|dt<\infty$ to guarantee the existence of the strong solution to \eqref{eq:sde}.

As shown in Fig.~\ref{fig:h0h1}, we consider the terminal condition such that $x(T)> Td$ unsafe, where $d>0$ is a given constant. When there is no attack (i.e., $\theta(t)=0, \forall t \in [0,T]$), we have $x(T)\sim\mathcal{N}(0, T)$. Therefore, the probability of the terminal state being unsafe is $\Phi(-\sqrt{T}d)$. This probability can be altered by injecting a non-zero attack signal. 

\begin{figure}[t]
\vspace{2mm}
\centering    \includegraphics[width=0.94\columnwidth]{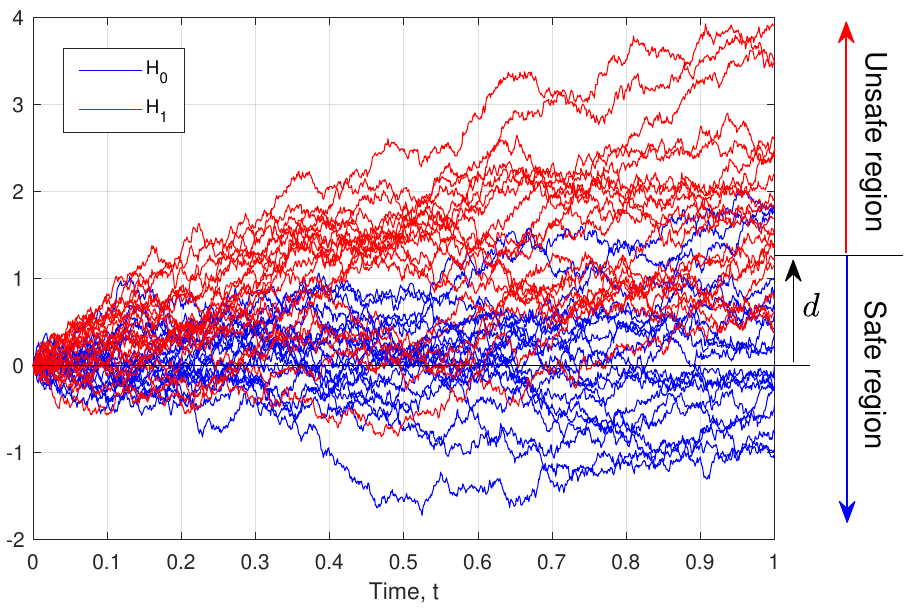}
    \caption{Sample paths of \eqref{eq:sde} with $\theta(t)=0$ (blue) and sample paths \eqref{eq:sde} with $\theta(t)=2$  (red). We assume $T=1$ in this plot. }
    \label{fig:h0h1}
    \vspace{-5mm}
\end{figure}

In this paper (except in Section~\ref{sec:feedback} where we consider feedback policies), we restrict the attacker's strategy to the class of open-loop policies.
That is, the attacker must fix an attack signal $\theta: [0,T]\rightarrow \mathbb{R}$ \textit{a priori} without observing the state $x(t)$.
We also restrict the attacker to the space of pure strategies, i.e., the action is not randomized.
With this setup, we have $x(T)\sim\mathcal{N}(m, T)$ with $m =\int_0^T \theta(t)dt$.
Consequently, the probability of the terminal state being unsafe is $\Phi(\frac{m}{\sqrt{T}}-\sqrt{T}d)$. In the sequel, we call $\gamma(\theta):=\Phi(\frac{m}{\sqrt{T}}-\sqrt{T}d)$ the attack success rate. 
We assume that, whenever the attack is applied, the attacker must ensure that the attack success rate is beyond a given threshold $c$, i.e., $\gamma(\theta) \geq c$.
Notice that requiring $\gamma(\theta) \geq c$ is equivalent to imposing a constraint
\begin{equation}
\label{eq:int_theta}
\int_0^T \theta(t)dt\geq \sqrt{T}\Phi^{-1}(c)+Td
\end{equation}
on the attack signal $\theta$.
In the sequel, we assume $c>\frac{1}{2}$ and $d>0$, which ensures that the quantity \eqref{eq:int_theta} is positive.

From the detector's viewpoint, it is not known in advance if the vehicle's operation is nominal or under attack.
Therefore, the detector's task is to determine if the observed trajectory $x$ is generated by \eqref{eq:sde} with $\theta=0$, or if a non-zero attack signal $\theta$ is injected. The former scenario is the null hypothesis $H_0$, whereas the latter is the alternative $H_1$:
\begin{align}
H_0 &: \theta(t)=0 \;\; \forall t \in [0,T] \label{eq:h0} \\
H_1 &: \int_0^T \theta(t)dt\geq \sqrt{T}\Phi^{-1}(c)+Td. \label{eq:h1} 
\end{align}
Notice that this is a composite hypothesis testing problem since, while the null hypothesis $H_0$ is simple, the alternative $H_1$ contains a family of functions $\theta$.

The role of the detector is to design a hypothesis testing algorithm  $\phi: C[0,T]\rightarrow \{0,1\}$ such that
\begin{equation}
\label{eq:phi}
\phi(x)=\begin{cases}
0 & \text{``No alarm"} \\
1 & \text{``Alarm"}
\end{cases}.
\end{equation}
The detector's decision \eqref{eq:phi} is made \textit{a posteriori} after observing the entire state trajectory  $x(t), 0\leq t\leq T$.

The quality of a testing algorithm $\phi$ is measured in terms of the probability $\alpha(\phi)$ of a false alarm (also known as Type I error) and the probability $\beta(\theta, \phi)$ of a detection failure (also known as Type II error):
\begin{align}
\alpha(\phi)&:= \text{Pr}\{\phi(x)=1 \;\; \big| \; H_0 \text{ is true}\} \\
\beta(\theta, \phi)&:= \text{Pr}\{\phi(x)=0 \;\; \big| \; H_1 \text{ is true}\}.
\end{align}
We say that a testing algorithm $\phi$ is admissible if $\alpha(\phi)\leq \epsilon$ for some given constant $\epsilon \in (0,\frac{1}{2})$.

In this paper, we model the interaction of the detector and the attacker as a zero-sum game.
Since the false alarm rate $\alpha(\phi)$ does not depend on the attacker's policy $\theta$, it is convenient to use $\alpha(\phi)$ as a constraint on the detector's policy.
Similarly, since the attack success rate $\gamma(\theta)$ does not depend on the detector's policy $\phi$, it is convenient to use $\gamma(\theta)\geq c$ as a constraint on the attacker's policy. 
In contrast, the detection failure rate $\beta(\theta,\phi)$ depends on both parties' policies which is minimized by the detector and maximized by the attacker.
Therefore, we formulate a mini-max game:
\begin{equation}
\label{eq:minimax}
p^*=\min_{\phi:\alpha(\phi)\leq \epsilon} \max_{\theta: \gamma(\theta)\geq c} \beta(\theta, \phi)
\end{equation}
and its dual:
\begin{equation}
\label{eq:maxmin}
d^*= \max_{\theta: \gamma(\theta)\geq c} 
\min_{\phi:\alpha(\phi)\leq \epsilon}
\beta(\theta, \phi).
\end{equation}
By the weak duality, $p^*\geq d^*$ holds trivially. 
In this paper, we will provide a unique pair of policies $(\theta^*, \phi^*)$ that constitutes a saddle point of the game, satisfying
\begin{equation}
\label{eq:saddle}
\beta(\theta, \phi^*) \leq \beta(\theta^*, \phi^*) \leq \beta(\theta^*, \phi)
\end{equation}
for all $\phi$ with $\alpha(\phi)\leq \epsilon$ and for all $\theta$ with $\gamma(\theta)\geq c$.
Consequently, the strong duality $p^*=d^*$ will be established, and the value of the game $\beta(\theta^*, \phi^*)$ will be computed.

\section{Preliminaries}
This section summarizes the mathematical ingredients needed to derive the main result.
\subsection{Girsanov's theorem}
Given two random processes $x$ and $w$ related by \eqref{eq:sde}, we have already defined the probability measure $\mu$ as the one in which $w$ is the standard Brownian motion. In $\mu$, $x$ is not the Brownian motion. However, Girsanov's theorem \cite[Theorem 8.6.3]{oksendal2003stochastic} \cite[Theorem 6.3]{liptser1977statistics} states that there exists an alternative measure $\mu_\theta$ in which $x$ is the standard Brownian motion. Specifically, for each sample path $x$, the likelihood ratio $\frac{d\mu}{d\mu_\theta}(x)$ is given by
\begin{subequations}
\label{eq:girsanov} 
\begin{align}
\frac{d\mu}{d\mu_\theta}(x)&=\exp\left\{\int_0^T \theta(t)dw(t)+\frac{1}{2}\int_0^T \theta^2(t)dt\right\} \label{eq:girsanov1} \\
&=\exp\left\{\int_0^T \theta(t)dx(t)-\frac{1}{2}\int_0^T \theta^2(t)dt\right\}. \label{eq:girsanov2}
\end{align}
\end{subequations}
That is, observing a particular sample path $x$ as an outcome of \eqref{eq:sde} (this occurs with probability $\propto \mu(x)$) is $\frac{d\mu}{d\mu_\theta}(x)$ times more likely than observing the same sample path $x$ as a realization of the standard Brownian motion (this occurs with probability $\propto \mu_\theta(x)$).

The false alarm rate $\alpha(\phi)$ is defined as the probability of $\phi(x)=1$ when there is no attack, i.e., when $x$ is the standard Brownian motion. Since $x$ is the standard Brownian motion in $\mu_\theta$, this quantity can be expressed as $\alpha(\phi)=\mathbb{E}^{\mu_\theta}\left[\phi(x)\right]$. Despite the appearance of $\theta$ on the right-hand side, the false alarm rate does not depend on $\theta$.
In contrast, the detection failure rate  $\beta(\theta, \phi)=1-\mathbb{E}^{\mu}\left[\phi(x)\right]$ depends on $\theta$, as the distribution of $x$ depends on $\theta$ under the measure $\mu$ in which $w$ is the standard Brownian motion.

\subsection{Neyman-Pearson lemma}
Consider the binary hypothesis testing problem in which the simple null hypothesis \eqref{eq:h0} is to be discriminated from the simple alternative \eqref{eq:h1} with a fixed $\theta$ satisfying $\gamma(\theta)\geq c$.
\begin{lemma}
\label{lem:np}
The testing algorithm $\phi: C[0,T]\rightarrow \{0,1\}$ that minimizes $\beta(\theta, \phi)$ subject to the constraint $\alpha(\phi)\leq \epsilon$ is given by
\begin{equation}
\label{eq:phi_lambda}
\phi(x)=\begin{cases}
0 & \text{ if } z_\theta(x, T) \leq \lambda^* \\
1 & \text{ if } z_\theta(x, T) > \lambda^* 
\end{cases} \quad \mu\text{-almost surely}
\end{equation}
where
\begin{equation}
z_\theta(x, t)=\exp\left\{\int_0^t \! \theta(s)dx(s)-\frac{1}{2}\int_0^t \!\theta^2(s)ds\right\}
\end{equation}
and $\lambda^*>0$ is a constant satisfying $\alpha(\phi)= \epsilon$.
\end{lemma}
\begin{proof}
We accept the Neyman-Pearson lemma \cite{cvitanic2001generalized}, which states that the optimal hypothesis test to discriminate a null hypothesis $x\sim \mu_\theta(x)$ from an alternative $x\sim \mu(x)$ is in general given by a randomized policy of the form
\begin{equation}
\phi(x)=\mathds{1}_{\{z_\theta(x, T)>\lambda^*\}}+b\cdot \mathds{1}_{\{z_\theta(x, T)=\lambda^*\}}
\end{equation}
$\mu$-almost surely, where $\lambda^*=\inf\{\lambda\geq 0: \alpha(\phi)\leq \epsilon\}$ and $b$ is an appropriate binary random variable.
To complete the proof based on this result, it is sufficient to show that $\lambda^*$ in our setup attains $\alpha(\phi)=\epsilon$, and $\text{Pr}\{z_\theta(x, T)=\lambda^*\}=0$.

Notice that 
\begin{equation}
\alpha(\phi)=\mathbb{E}^{\mu_\theta}\left[\phi(x)\right]=\mathbb{E}^{\mu_\theta}\left[ \mathds{1}_{\left\{z_\theta(x, T)>\lambda\right\}}\right].
\end{equation}
Since $x$ is the standard Brownian motion in $\mu_\theta$, $\alpha(\phi)$ can be written in term of the standard Brownian motion $x$ as
\begin{align}
&\alpha(\phi)=\text{Pr}\left\{
\exp\left\{\int_0^T \theta(t)dx(t)-\frac{1}{2}\int_0^T \theta^2(t)dt\right\}>\lambda
\right\} \nonumber \\
&=\text{Pr}\left\{
\int_0^T \theta(t)dx(t)>\frac{1}{2}\int_0^T \theta^2(t)dt+\log\lambda
\right\}.
\end{align}
Since the random variable $X:=\int_0^T \theta(t)dx(t)$ has a continuous cumulative distribution function, $\alpha(\phi)$ is a continuous, non-increasing function of $\lambda$ such that $\alpha(\phi)\rightarrow 1$ as $\lambda\rightarrow 0$ and $\alpha(\phi)\rightarrow 0$ as $\lambda\rightarrow +\infty$.
Therefore, for each $\epsilon\in(0,\frac{1}{2})$, there exists $\lambda^*>0$, which is the smallest constant satisfying $\alpha(\phi)\leq \epsilon$.
Using the fact that $X$ has a continuous cumulative distribution function, we also have $\text{Pr}\{z_\theta(x, T)=\lambda^*\}=\text{Pr}\{X=\frac{1}{2}\int_0^T \theta^2(t)dt+\log\lambda^* \} =0.$
\end{proof}

For a fixed $\theta$ such that $\gamma(\theta)\geq c$, let $\phi$ be the optimal test given by Lemma~\ref{lem:np}.
Then, the detection failure rate can be written as
\begin{align}
\beta(\theta,\phi)&=1-\mathbb{E}^\mu \left[\phi(x)\right] \nonumber \\
&=1-\mathbb{E}^{\mu_\theta} \left[\phi(x)z_\theta(x,T)\right] \nonumber \\
&=1-\lambda^* \mathbb{E}^{\mu_\theta} \left[\phi(x)\right]-\mathbb{E}^{\mu_\theta} \left[\phi(x)(z_\theta(x,T)-\lambda^*)\right] \nonumber \\
&=1-\lambda^* \epsilon -\mathbb{E}^{\mu_\theta}\left[z_\theta(x,T)-\lambda^*\right]^+ \label{eq:beta_simple}
\end{align}
In the last step, we used the fact that $\mathbb{E}^{\mu_\theta} \left[\phi(x)\right]=\alpha(\phi)=\epsilon$ (Lemma~\ref{lem:np}) and \eqref{eq:phi_lambda}.

\section{Main Result}
The main result of this paper is the following:
\begin{theorem}
\label{thm:main}
The following pair of policies form a saddle point of the zero-sum game \eqref{eq:minimax} and \eqref{eq:maxmin}:
\begin{align}
\theta^*(t)&=\bar{\theta}:=\frac{1}{\sqrt{T}}\Phi^{-1}(c)+d \;\; \forall t\in [0,T] \label{eq:theta_star}\\
\phi^*(x)&=
\begin{cases}
0 & \text{ if } x(T)\leq \sqrt{T}\Phi^{-1}(1-\epsilon) \\
1 & \text{ if } x(T)> \sqrt{T}\Phi^{-1}(1-\epsilon). \label{eq:phi_star}
\end{cases}
\end{align}
Moreover, the saddle point $(\theta^*,\phi^*)$ is unique in the sense that if $(\theta', \phi')$ is another saddle point, then $\theta^*(t)=\theta'(t)$ holds almost everywhere in $[0,T]$ and $\phi^*(x)=\phi'(x)$ holds $\mu$-almost surely.
Furthermore, the value of the game is 
\begin{equation}
\label{eq:value}
\beta(\theta^*, \phi^*)=\Phi(\Phi^{-1}(1-\epsilon)-\Phi^{-1}(c)-\sqrt{T}d).
\end{equation}
Specifically, if $1-\epsilon=c$, then $\beta(\theta^*, \phi^*)=\Phi(-\sqrt{T}d)$.
\end{theorem}
\begin{remark}
Theorem~\ref{thm:main} states that the max-min policy (the most covert attack) is a constant bias injection $\theta(t)=\bar{\theta}$, where the constant $\bar{\theta}$ is chosen to be the smallest value satisfying $\gamma(\theta)\geq c$. 
Conversely, the minimax policy $\phi^*(x)$ (i.e., the most powerful hypothesis test) only examines the final value $x(T)$ of the observed sample path $x$.
As we will see below, $\phi^*(x)$ can be viewed as the Neyman-Pearson type binary hypothesis testing algorithm that discriminates $H_0: dx_t=dw_t$ from $H_1: dx_t=\bar{\theta}dt+dw_t$.
\end{remark}
\begin{remark}
Suppose both $(\theta^*,\phi^*)$ and $(\theta',\phi')$ are saddle points of the game. Then, Theorem~\ref{thm:main} states that they can only differ on a set with measure zero.
By the interchangeability of saddle points of two-person zero-sum games \cite{bacsar1998dynamic}, it also follows that $(\theta^*, \phi^*)$, $(\theta^*, \phi')$, $(\theta', \phi^*)$, and $(\theta', \phi')$ are all saddle points, and they attain the same value. 
\end{remark}

To prove Theorem~\ref{thm:main}, notice that for any fixed $\theta^*$ such that $\gamma(\theta^*)\geq c$ and $\phi^*$ such that $\alpha(\phi^*)\leq \epsilon$, we have
\begin{subequations}
\label{eq:inf_sup_thm}
\begin{align}
\inf_{\phi:\alpha(\phi)\leq \epsilon} \beta(\theta^*, \phi)
&\leq \sup_{\theta:\gamma(\theta)\geq c}\inf_{\phi:\alpha(\phi)\leq \epsilon} \beta(\theta, \phi) \label{eq:inf_sup_thm1} \\
&\leq \inf_{\phi:\alpha(\phi)\leq \epsilon}\sup_{\theta:\gamma(\theta)\geq c} \beta(\theta, \phi) \label{eq:inf_sup_thm2}\\
&\leq \sup_{\theta:\gamma(\theta)\geq c} \beta(\theta, \phi^*) \label{eq:inf_sup_thm3}
\end{align}
\end{subequations}
where \eqref{eq:inf_sup_thm2} follows from the max-min inequality.
Hence, if the pair $(\theta^*, \phi^*)$ satisfies the saddle point condition
\begin{equation}
\label{eq:saddle2}
\sup_{\theta:\gamma(\theta)\geq c}\beta(\theta, \phi^*) = \beta(\theta^*, \phi^*) = \inf_{\phi:\alpha(\phi)\leq \epsilon} \beta(\theta^*, \phi)
\end{equation}
then the chain of inequalities \eqref{eq:inf_sup_thm} holds with equality and the strong duality 
$p^*=d^*=\beta(\theta^*, \phi^*)$
is implied. 

Therefore, in Subsection~\ref{sec:saddle} below, we show that the pair $(\theta^*, \phi^*)$ given by \eqref{eq:theta_star} and \eqref{eq:phi_star} indeed satisfies the saddle point condition \eqref{eq:saddle2}.
However, such an argument is insufficient to prove that $(\theta^*, \phi^*)$ is the unique saddle point.
To establish the uniqueness result, notice that the first inequality \eqref{eq:inf_sup_thm1} implies that if $(\theta', \phi')$ is a saddle point, then $\theta'$ must be the max-min solution that attains
\begin{equation}
\label{eq:theta_unique}
\sup_{\theta:\gamma(\theta)\geq c}\inf_{\phi:\alpha(\phi)\leq \epsilon} \beta(\theta, \phi) = \inf_{\phi:\alpha(\phi)\leq \epsilon} \beta(\theta', \phi).
\end{equation}
and that $\phi'$ is the best response to $\theta'$.
In Subsection~\ref{sec:unique} below, we show that \eqref{eq:theta_unique} is attained uniquely by $\theta'=\theta^*$, and that $\phi^*$ is the unique best response to $\theta^*$.
This will establish the uniqueness of the saddle point $(\theta^*, \phi^*)$.

\subsection{Saddle point condition}
\label{sec:saddle}
We prove \eqref{eq:saddle2} by showing the first equality (optimality of $\theta^*$) and the second equality (optimality of $\phi^*$) separately. 

\subsubsection{Optimality of $\theta^*$}
We first prove that $\beta(\theta,\phi^*)\leq \beta(\theta^*,\phi^*)$ holds for all $\theta$ such that $\gamma(\theta)\geq c$.
Notice that the function $\phi^*(x)$ in \eqref{eq:phi_star} only depends on the terminal state $x(T)$. Moreover, under any admissible attack strategy, we have $x(T)\sim\mathcal{N}(m,T)$, where $m=\int_0^T \theta(t)dt \geq \sqrt{T}\Phi^{-1}(c)+Td$.
Therefore, to maximize the detection failure rate, it is optimal for the attacker to choose a strategy that attains the smallest admissible value of $m$. Hence, any function $\theta: [0,T]\rightarrow \mathbb{R}$ such that $\int_0^T \theta(t)dt = \sqrt{T}\Phi^{-1}(c)+Td$ is a best response to $\phi^*$.
Since such a class of functions contains $\theta^*$, we have
$\beta(\theta,\phi^*)\leq \beta(\theta^*,\phi^*)$.

\subsubsection{Optimality of $\phi^*$}
We next prove that $\beta(\theta^*, \phi^*) \leq \beta(\theta^*, \phi)$ holds for all $\phi$ such that $\alpha(\phi)\leq \epsilon$.
To this end, let the attacker's policy be fixed to $\theta^*$ in \eqref{eq:theta_star}.
Then, by Lemma~\ref{lem:np}, the optimal test $\phi$ is given by
\begin{equation}
\label{eq:phi_lambda_bar}
\phi(x)=\begin{cases}
0 & \text{ if } z_{\theta^*}(x,T) \leq \lambda^* \\
1 & \text{ if } z_{\theta^*}(x,T) > \lambda^* 
\end{cases}
\end{equation}
where 
\begin{subequations}
\begin{align}
z_{\theta^*}(x, T)&=\exp\left\{\int_0^T \bar{\theta}dx(t)-\frac{1}{2}\int_0^T{\bar{\theta}}^2 dt\right\} \\
&=\exp\left\{\bar{\theta}x(T)-\frac{T}{2}{\bar{\theta}}^2\right\}.
\end{align}
\end{subequations}
Hence, it is sufficient to show that \eqref{eq:phi_lambda_bar} is equivalent to \eqref{eq:phi_star}.

Since $x(T)\sim \mathcal{N}(0,T)$ under $\mu_\theta$,
\begin{subequations}
\begin{align}
\alpha(\phi)&=\mathbb{E}^{\mu_\theta}\left[\phi(x)\right] \\
&=\mu_\theta \left(\left\{\exp\left(\bar{\theta}x(T)-\frac{T}{2}{\bar{\theta}}^2\right)>\lambda^*\right\}\right) \\
&=1-\Phi\left(\frac{\sqrt{T}}{2}\bar{\theta}+\frac{\log \lambda^*}{\sqrt{T}\bar{\theta}}\right).
\end{align}
\end{subequations}
Solving $\alpha(\phi)=\epsilon$, we obtain
\begin{equation}
\label{eq:lambda_opt}
\lambda^*=\exp\left(\sqrt{T}\bar{\theta}\Phi^{-1}(1-\epsilon)-\frac{T}{2}\bar{\theta}^2\right).
\end{equation}
Substituting \eqref{eq:lambda_opt}
into
\eqref{eq:phi_lambda_bar}, we obtain \eqref{eq:phi_star}.

\subsection{Uniqueness of the saddle point}
\label{sec:unique}
We now solve the max-min problem on the left-hand side of \eqref{eq:theta_unique}.
Let the function $\theta: [0,T]\rightarrow \mathbb{R}$ be fixed.
Then, according to the Neyman-Pearson lemma, the best response is a threshold-based policy of the form:
\begin{equation}
\label{eq:phi_lambda2}
\phi(x)=\begin{cases}
0 & \text{ if } z_\theta(x,T) \leq \lambda^* \\
1 & \text{ if } z_\theta(x,T) > \lambda^* 
\end{cases}
\end{equation}
where
\begin{equation}
z_\theta(x,t)=\exp\left\{\int_0^t \! \theta(s)dx(s)-\frac{1}{2}\int_0^t \!\theta^2(s)ds\right\}
\end{equation}
and $\lambda^*>0$ is a constant that satisfies $\alpha(\theta)= \epsilon$.
Assuming the best response \eqref{eq:phi_lambda2} to the attack signal $\theta$, we obtain from \eqref{eq:beta_simple} that
\begin{align}
&\sup_{\theta:\gamma(\theta)\geq c}\inf_{\phi:\alpha(\phi)\leq \epsilon} \beta(\theta, \phi) \nonumber \\
&=\sup_{\theta:\gamma(\theta)\geq c}
1-\lambda^* \epsilon -\mathbb{E}^{\mu_\theta}\left[z_\theta(x,T)-\lambda^*\right]^+. \label{eq:sup_theta}
\end{align}
We will show that this supremum is attained by a constant function $\theta^*(t)=\bar{\theta}$ given by \eqref{eq:theta_star}, and that any $\theta$ that attains the supremum must coincide with \eqref{eq:theta_star} almost everywhere in $t\in [0,T]$.
Recall that the last term in \eqref{eq:sup_theta} means
\begin{align}
&\mathbb{E}^{\mu_\theta}\left[z_\theta(x,T)-\lambda^*\right]^+ \nonumber \\
&=\!\mathbb{E}^{\mu_\theta}\!\!\left[\exp\left\{\int_0^T\theta(t)dx(t)\!-\!\frac{1}{2}\!\int_0^T \!\!\theta^2(t)dt\right\}\!-\!\lambda^*\right]^+
\label{eq:inf_max_x}
\end{align}
Since $x(t)$ is the standard Brownian motion in $\mu_\theta$, and since $w(t)$ is the standard Brownian motion in $\mu$, the quantity \eqref{eq:inf_max_x} can also be written as
\begin{equation}
\label{eq:inf_max_w}
\mathbb{E}^{\mu}\left[\exp\left\{\int_0^T \! \theta(t)dw(t)\!-\!\frac{1}{2}\int_0^T \!\!\theta^2(t)dt\right\}\!-\!\lambda^*\right]^+ \!\!\!.
\end{equation}
Introducing
\begin{equation}
\label{eq:def_zeta}
\zeta_\theta(t):=\exp\left\{\int_0^t\theta(s)dw(s)-\frac{1}{2}\int_0^t \theta^2(s)ds\right\},
\end{equation}
\eqref{eq:inf_max_w} can also be written as $\mathbb{E}^\mu\left[\zeta_\theta(T)-\lambda^*\right]^+$.
Therefore, it is left to prove the following lemma:
\begin{lemma}
Let $\theta^*$ be given by \eqref{eq:theta_star}. Then,
\begin{equation}
\label{eq:lem_theta_opt}
\inf_{\theta:\gamma(\theta)\geq c}\mathbb{E}^{\mu}\left[\zeta_\theta(T)-\lambda^*\right]^+=\mathbb{E}^{\mu}\left[\zeta_{\theta^*}(T)-\lambda^*\right]^+.
\end{equation}
Moreover, if $\theta'$ attains the infimum on the left-hand side, then $\theta'(t)=\theta^*(t)$ almost everywhere in $t\in[0,T]$.
\end{lemma}
\begin{proof}
The proof strategy is inspired by \cite[Section 5]{cvitanic2001generalized}, which is further attributed to \cite{xu1992duality}.

Let $f:\mathbb{R}\rightarrow [0,\infty)$ be a convex function satisfying the linear growth condition. We will show that
\begin{equation}
\label{eq:lem_theta_opt_gen}
\mathbb{E}^\mu f(\zeta_{\theta^*}(T)) \leq
\mathbb{E}^\mu f(\zeta_{\theta}(T))
\end{equation}
for all $\theta$ such that $\gamma(\theta)\geq c$.
The claim \eqref{eq:lem_theta_opt} follows from 
\eqref{eq:lem_theta_opt_gen} by choosing $f(z)=[z-\lambda^*]^+$.

For each $\theta$ such that $\gamma(\theta)\geq c$, observe that 
$\Lambda(t):=\int_0^t \theta^2(\tau)/\bar{\theta}^2 d\tau$
is a non-decreasing function.
Moreover, consider the inner product $\left<v_1, v_2 \right>:=\int_0^T v_1(t)v_2(t)dt$ of $v_1(t)=\theta(t)$ and $v_2(t)=1$.
Since $
\|v_1\|^2=\int_0^T \theta^2(t)dt, \;\|v_2\|^2=T,
$
and $\gamma(\theta)\geq c$ implies $
\left<v_1, v_2\right>=\int_0^T \theta(t)dt \geq T\bar{\theta}$,
it follows from the Cauchy-Schwarz inequality that
$\int_0^T \theta^2(t)dt \geq T\bar{\theta}^2$. The equality holds if and only if $\theta(t)=\bar{\theta}$ almost everywhere.
Hence, we have $\Lambda(T) \geq T$.
Therefore, if we define a right inverse $\Lambda^{-1}$ of $\Lambda$ by
$\Lambda^{-1}(s):=\inf \left\{t: \Lambda(t) >s\right\}$,
we have $\Lambda^{-1}(T)\leq T$.
Now, notice that a time-changed process
\begin{equation}
\hat{w}(s):=\int_0^{\Lambda^{-1}(s)} \frac{\theta(s)}{\bar{\theta}}dw(s), \;\; 0\leq s \leq T
\end{equation}
is a martingale of the filtration $\hat{\mathcal{F}}(s):=\mathcal{F}(\Lambda^{-1}(s)), \; 0\leq s \leq T$ such that
\vspace{-1ex}
\begin{equation}
\mathbb{E}\hat{w}^2(s)=\int_0^{{\Lambda}^{-1}(s)} \frac{\theta^2(s)}{\bar{\theta}^2}ds=\Lambda(\Lambda^{-1}(s))=s.
\end{equation}
Therefore, $\hat{w}(s)$ is a Brownian motion with respect to $\hat{\mathcal{F}}(s)$.
Moreover, considering the time change $\tau=\Lambda^{-1}(\sigma)$,
\begin{subequations}
\label{eq:sde_zeta_theta}
\begin{align}
&1+\int_0^s \zeta_\theta (\Lambda^{-1}(\sigma)) \bar{\theta}d\hat{w}(\sigma) \\
&=1+
\int_0^{\Lambda^{-1}(s)} \zeta_\theta (\tau) \bar{\theta}d\hat{w}(\Lambda(\tau)) \\
&=1+
\int_0^{\Lambda^{-1}(s)} \zeta_\theta (\tau) \bar{\theta}\frac{\theta(\tau)}{\bar{\theta}}dw(\tau) \\
&=1+
\int_0^{\Lambda^{-1}(s)} \zeta_\theta (\tau) \theta(\tau)dw(\tau) \\
&=\zeta_\theta(\Lambda^{-1}(s)). 
\end{align}
\end{subequations}
In the last step, we used the fact that $\zeta_\theta(t)$ as defined in \eqref{eq:def_zeta} satisfies the stochastic differential equation $d\zeta_\theta(t)=\zeta_\theta(t)\theta(t)dw(t)$ (see, e.g., \cite[Exercise 4.4]{oksendal2003stochastic}). Since
\begin{equation}
\label{eq:sde_zeta_theta_bar}
\zeta_{\theta^*}(t)=1+\int_0^t \zeta_{\theta^*}(\tau) \bar{\theta}dw(\tau),
\end{equation}
comparing \eqref{eq:sde_zeta_theta} and \eqref{eq:sde_zeta_theta_bar}, we conclude that processes $\zeta_{\bar{\theta}}(\cdot)$ and $\zeta_{\theta}(\Lambda^{-1}(\cdot))$ have the same distribution.
From this observation, and from the optional sampling theorem, we have
\begin{equation}
\label{eq:equiv_d}
\mathbb{E}^\mu f(\zeta_{\theta^*}(t))=\mathbb{E}^\mu f(\zeta_{\theta}(\Lambda^{-1}(t)))
\end{equation}
for $0\leq t \leq T$. Also, since $\Lambda^{-1}(T)\leq T$, and since $f(\zeta_\theta(\cdot))$ is a submartingale (a consequence of Jensen's inequality), we obtain
\begin{equation}
\label{eq:submartingale}
\mathbb{E}^\mu f(\zeta_{\theta}(\Lambda^{-1}(T))) \leq
\mathbb{E}^\mu f(\zeta_\theta(T)).
\end{equation}
From \eqref{eq:equiv_d} and \eqref{eq:submartingale}, we obtain \eqref{eq:lem_theta_opt_gen}.
\end{proof}

\subsection{Value of the game}
\label{sec:value}
The result \eqref{eq:value} follows directly as follows:
\begin{subequations}
\begin{align}
&\beta(\theta^*, \phi^*)
=\mathbb{E}^\mu \left[1-\phi(x)\right] \\
&=\mathbb{E}^\mu \left[\mathds{1}_{\left\{x(T)\leq \sqrt{T}\Phi^{-1}(1-\epsilon)\right\}}\right] \\
&=\frac{1}{\sqrt{2\pi}}\int_{-\infty}^{\Phi^{-1}(1-\epsilon)} \exp\left\{-\frac{1}{2}(x-\sqrt{T}\bar{\theta})^2\right\}dx \label{eq:val_game3}\\
&=\Phi\left(\Phi^{-1}(1-\epsilon)-\sqrt{T}\bar{\theta}\right) \\
&=\Phi\left(\Phi^{-1}(1-\epsilon)-\Phi^{-1}(c)-\sqrt{T}d\right).
\end{align}
\end{subequations}
We used the fact that $\frac{1}{\sqrt{T}}x(T)\sim\mathcal{N}(\sqrt{T}\bar{\theta},1)$ in step \eqref{eq:val_game3}.

\section{Discussion}
\subsection{Feedback information structure}
\label{sec:feedback}
The results so far are restricted to games \eqref{eq:minimax} and \eqref{eq:maxmin} in which the attacker must choose the attack signal $\theta(t)$ in an open-loop manner. 
We now consider a modified setup in which the attacker is allowed to choose a state-dependent attack signal.
Specifically, in \eqref{eq:sde}, suppose that $\theta(t)$ is an $\mathcal{F}(t)$-adapted function satisfying $\text{Pr}\{\int_0^T |\theta(t)|dt\leq \infty\}=1$. Such a class of functions includes feedback policies and gives an advantage to the attacker.
We keep the strategy space for the detector the same. The detector's policy is a hypothesis testing algorithm $\phi: C[0,T]\rightarrow \{0,1\}$.

We now demonstrate that the pair of policies $(\theta^*, \phi^*)$ provided in Theorem~\ref{thm:main} is no longer a saddle point in this modified information structure. 
To see this, it is sufficient to construct a feedback policy $\theta'$ such that $\beta(\theta', \phi^*)>\beta(\theta^*, \phi^*)$.
To be concrete, assume $T=1$, $d=1.5$, and $c=1-\epsilon=0.95$.
In this case, the region $x(1)>1.5$ is considered unsafe and $\phi^*$ triggers the alarm if $x(1)\geq \Phi^{-1}(1-\epsilon)\approx 1.645$ as shown in Fig.~\ref{fig:bridge}.
Consider a feedback policy $\theta'(t)=\frac{b-x(t)}{1-t}$, where $b$ is a constant satisfying $d<b<\Phi^{-1}(1-\epsilon)$. In this case, \eqref{eq:sde} becomes
\begin{equation}
\label{eq:bridge}
dx(t)=\frac{b-x(t)}{1-t}dt+dw(t), \;\; x(0)=1.
\end{equation}
The solution to \eqref{eq:bridge} is known as the Brownian bridge, and satisfies $\lim_{t\rightarrow 1} x(t)=b$ $\mu$-almost surely (Fig.~\ref{fig:bridge}).
This means that the feedback policy $\theta'$ attains $\gamma(\theta')=1$ and $\beta(\theta',\phi^*)=1$. 
That is, the attacker wins most dramatically.

We are currently unaware of the saddle point strategies under the feedback information structure.

\begin{figure}[t]
\vspace{2mm}
\centering    \includegraphics[width=\columnwidth]{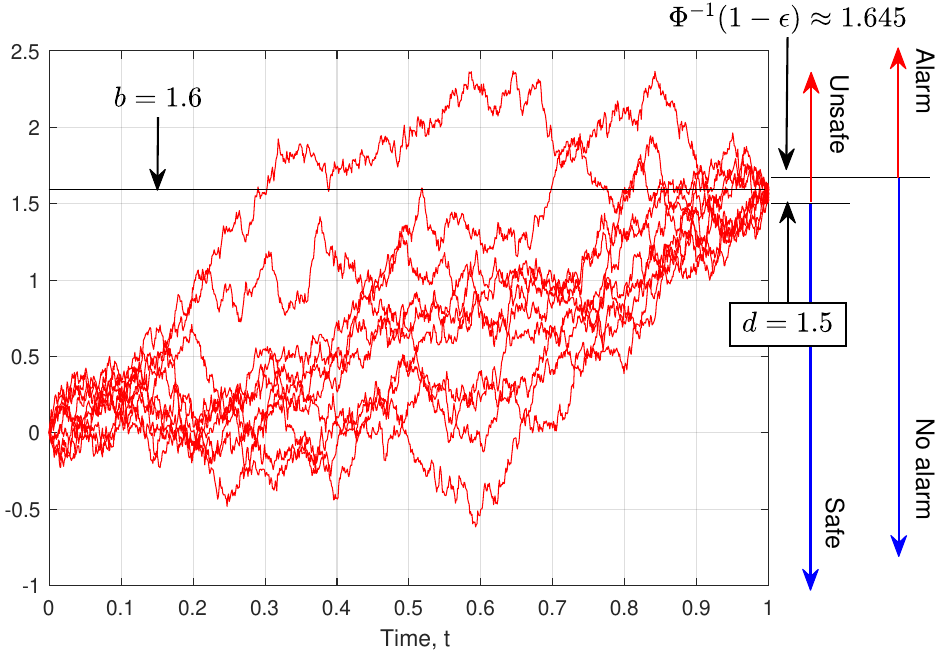}
\vspace{-5mm}
    \caption{Sample paths of the Brownian bridge. }
    \label{fig:bridge}
    \vspace{-7mm}
\end{figure}

\subsection{The first and the second order error exponents}

It is apparent from \eqref{eq:value} that the detection failure rate $\beta(\theta^*,\phi^*)$ diminishes to zero as $T\rightarrow \infty$.
While \eqref{eq:value} already provides a compact formula, it is insightful to characterize it in terms of the first and the second-order error exponents (i.e., the coefficients of the $T$ and $\sqrt{T}$ terms in $\log\beta(\theta^*, \phi^*)$).
To this end, let the attack signal $\theta^*(t)=\bar{\theta}$ be fixed.
Introduce the relative entropy rate $\bar{D}(\mu\|\mu_{\theta^*}):=\limsup_{T\rightarrow \infty}\frac{1}{T}D(\mu\|\mu_{\theta^*})$ and the variance rate $\bar{V}(\mu\|\mu_{\theta^*}):=\limsup_{T\rightarrow \infty}\frac{1}{T}V(\mu\|\mu_{\theta^*})$, where
\begin{align}
D(\mu\|\mu_{\theta^*})&:= \mathbb{E}^{\mu} \left[\log z_{\theta^*}\right] \\
V(\mu\|\mu_{\theta^*})&:=\mathbb{E}^\mu\left[\left(D(\mu\|\mu_{\theta^*})-\log z_{\theta^*}\right)^2 \right] 
\end{align}
with $\log z_{\theta^*}=\int_0^T \theta^*(t)dw(t)+\frac{1}{2}\int_0^T \theta^*(t)^2 dt$.
Using $\theta^*(t)=\bar{\theta}$ and the fact that $w(t)$ is the standard Brownian motion in $\mu$, we obtain
\vspace{-1ex}
\begin{equation}
\label{eq:d_v_rate}
\bar{D}(\mu\|\mu_{\theta^*})=\frac{1}{2}\bar{\theta}^2, \;\;
\bar{V}(\mu\|\mu_{\theta^*})=\bar{\theta}^2.
\end{equation}
Now, using \eqref{eq:phi_lambda_bar} and
\eqref{eq:lambda_opt}, $\beta(\theta^*,\phi^*)$ can be written as
\vspace{-1ex}
\begin{align}
&\beta(\theta^*,\phi^*)=\mu\left(\left\{\log z_{\theta^*} \leq \log \lambda^*\right\}\right) \nonumber \\ 
&=\mu\left(\left\{ \bar{\theta}w(T)+\tfrac{T}{2}\bar{\theta}^2 \leq \sqrt{T}\bar{\theta}\Phi^{-1}(1-\epsilon)-\tfrac{T}{2}\bar{\theta}^2\right\}\right) \nonumber \\
&=\mu\left(\left\{ 
\tfrac{w(T)}{\sqrt{T}}\leq \Phi^{-1}(1-\epsilon)-\sqrt{T}\bar{\theta}
\right\}\right) \nonumber \\
&=\mu\left(\left\{ 
\tfrac{w(T)}{\sqrt{T}}\geq \sqrt{T}\bar{\theta}+\Phi^{-1}(\epsilon)
\right\}\right).
\end{align}
Since $w(T)/\sqrt{T}\sim\mathcal{N}(0,1)$ in $\mu$, by Hoeffding's inequality for sub-Gaussian distributions, we have $\beta(\theta^*, \phi^*)\leq \exp\{-\frac{1}{2}(\sqrt{T}\bar{\theta}+\Phi^{-1}(\epsilon))^2\}$, or
\vspace{-1ex}
\begin{equation}
-\log \beta(\theta^*, \phi^*)\geq \frac{1}{2}T\bar{\theta}^2+\sqrt{T}\bar{\theta}\Phi^{-1}(\epsilon)+\text{const.}
\end{equation}
Using \eqref{eq:d_v_rate}, this inequality can be expressed as
\vspace{-1ex}
\begin{align}
&-\log \beta(\theta^*, \phi^*) \nonumber \\
&\geq T\bar{D}(\mu\|\mu_{\theta^*})+\sqrt{T}\sqrt{\bar{V}(\mu\|\mu_{\theta^*})}\Phi^{-1}(\epsilon)+\text{const.}
\label{eq:2nd_asymptote}
\end{align}
The appearance of $\bar{D}(\mu\|\mu_{\theta^*})$ on the right-hand side of \eqref{eq:2nd_asymptote} is a reminiscent of Stein's lemma \cite[Theorem 11.8.3]{cover1999elements}.
The inequality above shows the achievability of the relative entropy rate as the first-order asymptotes. 
The second term provides a tighter estimate in the regime of finite $T$. 
Notably, \eqref{eq:2nd_asymptote} is consistent with the known characterization of the second-order asymptotes (e.g., \cite{li2014second,watanabe2018second,lungu2024optimal}), despite the major difference between prior works on discrete-time hypothesis tests and our study on continuous-time counterparts. 
The appearance of higher-order terms in \eqref{eq:2nd_asymptote} implies that stealthiness measured by relative entropy alone \cite{bai2017data, guo2018worst,shang2021worst} may not be accurate for moderate values of $T$.


\section{Future work}
While the scope of this paper is restricted to a simple system model \eqref{eq:sde}, the approach we introduced in this paper can be generalized to high-dimensional and nonlinear system models.
By substituting the function $\gamma(\theta)$ with other cost functions, the proposed framework accommodates a broader class of attack scenarios.
Numerical approaches to compute the saddle point solutions (e.g., \cite{patil2023simulator}) in these generalized settings are important research topics in the future.
Saddle point solutions under the feedback information structure need further investigation.
Finally, non-asymptotic (finite sample) analysis of the saddle point value for a broader class of games in view of the recent progress \cite{li2014second,watanabe2018second,lungu2024optimal} in information theory literature will also be a fruitful research direction.

\bibliographystyle{IEEEtran}
\bibliography{references}

\end{document}